%% file: main.tex
\documentclass[10pt]{article}
\usepackage{amsmath} 
\usepackage{lmodern}
\usepackage[T1]{fontenc}
\usepackage[english]{babel}
\usepackage[utf8]{inputenc}
\usepackage{amssymb}
\usepackage[hypcap]{caption}
\usepackage{amsmath, amsthm, amsfonts, amscd}
\usepackage[nottoc,numbib]{tocbibind}
\usepackage{xcolor}
\usepackage{ latexsym }
\usepackage{mathrsfs}
\usepackage{float}
\usepackage{bm}
\usepackage{eufrak}
\usepackage{graphicx}
\usepackage[labelformat=empty]{caption}
\usepackage{microtype}
\usepackage{todonotes}
\usepackage{makeidx}
\usepackage{verbatim}
\usepackage{hyperref}
\hypersetup{
colorlinks=false,
linkbordercolor=red,
pdfborderstyle={/S/U/W 1}
}
\usepackage{mdwlist}
\usepackage{mathabx}
\usepackage{mathtools}
\usepackage{titlesec}

\newtheorem{thm}{\scshape{Theorem}}[section]
\newtheorem{cor}[thm]{\scshape{Corollary}}
\newtheorem{prop}[thm]{\scshape{Proposition}}
\newtheorem{lemma}[thm]{\scshape{Lemma}}

\newtheorem*{theorem*}{Theorem}

\newtheorem*{prop*}{Proposition}
\newtheorem*{cor*}{Corollary}

\theoremstyle{definition}
\newtheorem{dfn}[thm]{\scshape{Definition}}

\newtheorem{example}[thm]{\scshape{Example}}

\newtheorem{remark}[thm]{\scshape{Remark}}

\def\subset{\subseteq}

\def\={\cong}

\def\mb{\mathbf}
\def\mc{\mathcal}
\def\mbb{\mathbb}

\def\ranlge{\rangle}
\def\l{\lambda}
\def\End{\it End}
\def\rank{\it rank}
\def\C{\langle C \rangle}
\def\tc{\textcolor}
\def\ab{\it ab}

\title{Random nilpotent groups,  polycyclic presentations,  and Diophantine problems}
\author{Albert Garreta, Alexei Miasnikov, and Denis Ovchinnikov}

\begin{document}

\maketitle

\begin{abstract}
We introduce a model of random f.g., torsion-free, $2$-step nilpotent groups (in short, $\tau_2$-groups). To do so, we show that these are precisely the groups that admit a presentation of the form $ \label{tau2pres_0}\langle A, C \mid [a_i, a_j]= \prod_t {\scriptstyle c_t^{\scriptscriptstyle \lambda_{t,i,j}}} \ (i< j), \ [A,C]=[C,C]=1\rangle,$  where $A=\{a_1, \dots, a_n\}$, and $C=\{c_1, \dots, c_m\}$. Hence, one may select a random $\tau_2$-group $G$ by fixing $A$ and $C$, and then randomly choosing exponents $\lambda_{t,i,j}$ with $|\lambda_{t,i,j}|\leq \ell$, for some $\ell$.

  We prove that, if $m\geq n-1\geq 1$, then the following holds asymptotically almost surely, as $\ell\to \infty$: The ring of integers $\mathbb{Z}$ is e-definable in $G$, systems of equations over $\mathbb{Z}$ are reducible to systems over $G$ (and hence they are undecidable), the maximal ring of scalars of $G$ is $\mathbb{Z}$, $G$ is indecomposable as a direct product of non-abelian factors, and $Z(G)=\langle C \rangle$. If, additionally, $m \leq n(n-1)/2$, then $G$ is regular (i.e. $Z(G)\leq {\it Is}(G')$). This is not the case if $m > n(n-1)/2$. 
  
  In the last section of the paper we introduce similar models of random polycyclic groups and random f.g. nilpotent groups of any nilpotency step, possibly with torsion. We quickly see, however, that the latter yields finite groups a.a.s. 

\end{abstract}

\hypersetup{linkcolor=black}

\tableofcontents
\hypersetup{linkcolor=gray}

\section{Introduction}
\label{introduction}

\input{intro}

\section{Preliminaries}

\input{prelim}

\input{everything_after_intro_and_prel}

\section{General polycyclic models}

\input{extended}

\bibliography{bib}

\end{document}

%% file: intro.tex

In \cite{Duchin2}, Cordes, Duchin, Duong, Ho, and Sánchez introduced a model of random finitely generated nilpotent groups. Such model is the analog of the \emph{few-relators} and the \emph{density}  models of random finitely presented groups, where one takes a free group $F_m=F_m(a_1, \dots, a_m)$,  and then adds  a set of random relations $R$. Every relator is chosen among all words of a certain length $\ell$ on the alphabet $A^{\pm 1}=\{a_1^{\pm 1}, \dots, a_m^{\pm 1}\}$, with uniform probability. The length $\ell$ is thought of as an integer variable that tends to infinity, and the  number  of chosen relators is taken to be a function of $\ell$. For instance, $|R| = (2m+1)^{d\ell}$ ($0<d<1$) in the density model, whereas $|R|$ is constant in the few-relators model.
One can then calculate the probability $p_{\ell}$ that a group $G=F_m/\langle \langle R\rangle \rangle$  satisfies some property $P$, for a fixed $\ell$. The limit $p=\lim_{\ell \to \infty} p_{\ell}$, if it exists, is called the \emph{asymptotic probability} that $G$ satisfies $P$. If $p=1$, then  $G$ is said to satisfy $P$ \emph{asymptotically almost surely} (a.a.s.)  For example, a well-known result of Gromov \cite{Gromov} states that, in the density model,  $G$ is hyperbolic  if $d<1/2$, and  finite  if $d > 1/2$, a.a.s. See \cite{Olivier} for more information on random f.p.\ groups.

Since all finitely generated nilpotent groups are quotients of free nilpotent groups  by some finite set of relators, one can easily adapt the procedure above  to the class of f.g.\ nilpotent groups: it suffices to replace  $F_m$ by an $s$-step rank-$m$ free nilpotent group $N_{s,m}=N_{s,m}(A)$. Then, as before, one  chooses a set $R$ of random words of length $\ell$ on the alphabet $A^{\pm 1}$. These words are added as relators to $N_{s,m}$, yielding a  random f.g.\ nilpotent group $G= N_{s, m}/ \langle \langle R \rangle \rangle$. An alternative model  was introduced in \cite{Delp}, where f.g.\ torsion-free nilpotent groups  are considered as subgroups of unitriangular matrices.   



In \cite{Part_3}, we studied the structure and the Diophantine problem of random f.g.\ nilpotent groups generated according to the model used in \cite{Duchin2}. With similar intentions, in this paper we restrict our attention to the class of f.g., torsion-free, $2$-step nilpotent groups, which we call $\tau_2$-groups. We first introduce a  model of random $\tau_2$-groups $G$, and then we apply some of the techniques developed in  \cite{Part_3}  to obtain information regarding their Diophantine problem and their structure. 

The approach we use  relies on the fact that all f.g.\ nilpotent groups are polycyclic, and therefore, that they admit a polycyclic presentation (see \cite{Handbook1}). One can thus select   $\tau_2$-groups by randomly specifying  polycyclic presentations $P$, making sure that the $P$'s are chosen in a way that the resulting groups are $\tau_2$-groups,  and that any $\tau_2$-group can be chosen this way. The following result, which we prove in Section \ref{randomSec}, allows one to do so. One would like to extend this approach to the  class of all (torsion-free) f.g.\ nilpotent groups, but, as far as we know, all reasonable generalizations yield finite groups a.a.s. In this paper, the \emph{rank} of an abelian group is the minimum cardinality of a set of generators. 
{
\renewcommand{\thethm}{\ref{SpanThm}}
\begin{thm}
Let $G$ be a group, and let $n,m$  be two integers greater or equal than zero. Then the following two statements are equivalent. 1)  $G$ admits a (polycyclic) presentation of the form
\begin{align}
G=\langle a_1, \dots, a_n, c_1, \dots, c_m \mid \ &[a_i, a_j]= \prod_{t=1}^m c_t^{\l_{t}^{ij}} \ \left(1\leq i<j\leq n\right),\label{tau2pres}\\ &[a_i,c_j]=[c_k,c_r]=1  \ \hbox{for all}\ i,j,k,r \rangle. \nonumber
\end{align}
2) $G$ is a $\tau_2$-group satisfying
 $\rank(G/Z(G)) + \rank(Z(G)) = n + m$,   and
 $\rank(G') \leq m \leq \rank(Z(G))$. 
\end{thm}
\addtocounter{thm}{-1}
}
%
%

%
Hence, by Theorem \ref{SpanThm}, a $\tau_2$-group  $G$ can be  chosen by randomly selecting a presentation of the form \eqref{tau2pres}, which we call a \emph{$\tau_2$-presentation}. To do so,  fix sets $A=\{a_1, \dots, a_n\}$, $C=\{c_1, \dots, c_m\}$, and an integer $\ell$, and then  randomly specify exponents $\l_{t}^{ij}$ such that $|\l_{t}^{ij}|\leq \ell$ (with uniform probability).  One then can study properties of the groups obtained this way as $\ell \to \infty$, in a similar fashion to what we explained previously. 

\begin{remark}[\textbf{Why   restrict to the class of $\tau_2$-groups}] 
In the last section of the paper we introduce a more natural and    general model of random f.g.\ nilpotent groups (of arbitrary nilpotency step, possibly with torsion). However, in Lemma \ref{nopol1}  we show that such model yields finite groups asymptotically almost surely. This is the reason why we have particularized it to $\tau_2$-groups.
 
In the same section we also provide a  model for the class of all polycyclic groups. In it,  these are obtained by randomly choosing polycyclic presentations. In this case it is proved that the resulting groups have finite abelianization a.a.s.\ (Lemma \ref{nopol2}). 
\end{remark}

One of the main aspects we study about random $\tau_2$-groups is their Diophantine problem:
%
%
%
\begin{dfn}\label{DiophDfn}
%
%
The \emph{Diophantine problem} over an algebraic structure $\mathcal{A}$, denoted $\mc{D}(\mc{A})$, refers to the task of determining what systems of equations over $\mc{A}$ ($\mathcal{A}$-systems) have solutions. An algorithm $L$ is said to \emph{solve} $\mc{D}(\mc{A})$ if, given an $\mc{A}$-system $S$,  determines whether $S$ has a solution or not. If such an algorithm exists, then
$\mc{D}(\mc{A})$ is called \emph{decidable}. Otherwise, $\mc{D}(\mc{A})$ is called  \emph{undecidable}.  Furthermore, $\mc{D}(\mc{A})$ is said to be \emph{reducible} to  $\mc{D}(\mc{M})$, for $\mc{M}$ another structure,  if a solution to $\mc{D}(\mc{M})$ (if it existed) could be used  as a subroutine to solve $\mc{D}(\mc{A})$. 
\end{dfn}
 
In this paper, a reduction always takes an $\mc{A}$-system $S$  as input, and  it constructs an  $\mc{M}$-system $S_{\mc{M}}$ that has a solution if and only if $S$ has. Notice that if $\mc{D}(\mathbb{Z})$ is reducible to $\mc{D}(\mc{M})$, then  $\mc{D}(\mc{M})$ is undecidable due to the negative answer to Hilbert's 10th problem, which states that  $\mc{D}(\mathbb{Z})$ is undecidable.  This idea was used in \cite{Romankov} (1979) and \cite{Duchin} (2014) to prove that the Diophantine problem is undecidable over a certain $4$-step nilpotent group, and over any non-abelian free nilpotent group, respectively. We refer to \cite{Duchin} for a survey of results on this topic. 

The following is one of the  main results of this paper:
{
\renewcommand{\thethm}{\ref{mainthm}}
\begin{thm}
Suppose $G$ is a random $\tau_2$-group obtained as above, with $m\geq n-1 \geq 1$. Then the following holds a.a.s.: $Z(G)=\C$, the ring $\mathbb{Z}$ is e-definable in $G$, $\mc{D}(\mathbb{Z})$ is reducible to $\mc{D}(G)$, $\mc{D}(G)$ is undecidable,  $\mathbb{Z}$ is the maximal ring of scalars of $G$, and $G$ is directly indecomposable into non-abelian factors.
\end{thm}
\addtocounter{thm}{-1}
}
Very roughly, a structure $\mc{A}=(A; f_1, \dots)$ is \emph{e-definable} in another structure $\mathcal{B}$ if the predicates "$x \in A$" and "$z\in A$ is in the image of $f_i$" can be effectively  expressed as systems of equations over $\mathcal{B}$ (see \ref{edfnDfn} for a formal definition). In this case, $\mc{D}(\mc{A})$ is reducible to $\mc{D}(\mc{B})$.  Hence, by the negative answer to Hilbert's 10th problem, $\mc{D}(\mc{B})$ is undecidable for any structure $\mc{B}$ in which  the ring $\mathbb{Z}$ is e-definable.

We also study the structure of  $Z(G)$ and $G'$, and we investigate whether  $Z(G)\leq {\it Is}(G')=\left\{ g \in G \mid g^t\in G' \ \hbox{for some} \ n\in \mathbb{Z}\backslash \{0\} \right\}$, in which case $G$ is called \emph{regular}:
{
\renewcommand{\thethm}{\ref{regularitythm}}
\begin{thm}
Let  $G$ be a  $\tau_2$-group $G$ obtained by randomly choosing a presentation $P\in \mc{P}(n,m,\ell)$ as in \eqref{tau2pres}, with $n\geq 2$ and $m\geq 1$. Then the following holds a.a.s.:
\begin{enumerate}
\item If $m\leq n(n-1)/2$, then $G'$ has finite index in $\C$.  If, additionally, $m\geq n-1$, then $G$ is regular.
\item If $m > n(n-1)/2$, then the set $\{[a_i, a_j] \mid i<j\}$ is a basis of $G'$, $G'$ has infinite index in $\C$, and $G$ is not regular. These last two properties hold always, and not only a.a.s.
\end{enumerate}
\end{thm}
\addtocounter{thm}{-1}
}

One of the main ingredients in the proofs of Theorems \ref{mainthm} and \ref{regularitythm} is the notion of c-small elements: We say that an element $g\in G$ is \emph{centralizer-small} (or just \emph{c-small}) if its  centralizer (the set of elements that commute with $g$) is $\{g^tz \mid t \in \mathbb{Z},\ z\in Z(G)\}$.  
It turns out that:
{
\renewcommand{\thethm}{\ref{mainThmEq2}}
\begin{thm}
$\mathbb{Z}$ is e-definable in any $\tau_2$-group having two non-commuting c-small elements.
\end{thm}
\addtocounter{thm}{-1}
}
The  techniques we use to prove Theorem \ref{mainThmEq2} have a resemblance with some arguments from Duchin, Liang, and Shapiro in \cite{Duchin}, and from Romankov in \cite{Romankov}.
%
%
%
%
%
%
%
%

Another important step towards proving Theorem \ref{mainthm} is given in  Theorem \ref{CSmallMaxZ}, where we show that if all the $a_i$'s of $G$ are c-small, then the maximal ring of scalars   of $G$ is $\mathbb{Z}$ (see Definition \ref{ringScalars}). A consequence of this fact  is that $G$ is indecomposable as a direct product of non-abelian subgroups (Proposition \ref{dirIndec}).

In views of these results, one can try to prove Theorem \ref{mainthm} by checking whether or not the $a_i$'s are c-small (notice, Item 1 of Theorem \ref{regularitythm} uses existence of c-small elements as well). We will see that this is the case, a.a.s., when $m\geq n-1$. The case $m< n-1$ remains open. The main difficulty in such case is to find an asymptotic description of $Z(G)$. Another problem that is left open  is that of determining what conclusions of  Theorems \ref{mainthm} and \ref{regularitythm} still hold when $m<n-1$, a.a.s.

%% file: prelim.tex
\subsection{Nilpotent groups}\label{nilpSub}
Following standard conventions, we call the element $[g,h]=g^{-1}h^{-1}gh$ of a group $G$  the \emph{commutator} of $g$ and $h$, and we denote the subgroup of $G$ formed by all its commutators by $G'$.  More generally, we define  inductively $G_1=G$, $G_2= G' =\langle[G,G]\rangle= \langle\{[g,h] \mid g,h\in G\}\rangle$, and  $G_{n+1} = \langle[ G, G_n]\rangle$.  The subnormal series $G_1\ \unrhd \ G_2\  \unrhd \dots$ is called the \emph{lower central series of $G$}. 
If $G_{s+1}=1$ for some $s$,  then $G$ is said to be  \emph{$s$-step nilpotent}, or just \emph{nilpotent}.  For example, $G$ is $1$-step nilpotent if and only if it is an abelian group, and it is $2$-step nilpotent if and only if $$G_3=\langle[G,[G, G]]\rangle=\langle\{[g_1,[g_2,g_3]] \ | \ g_i \in G\}\rangle=1,$$ which   is the same as saying that all  commutators belong to the center  of $G$, i.e.\ to $Z(G) = \{g \in G \mid [g,h]=1 \ \hbox{for all} \ h\in G\}$. The element $[g_1, [g_2, g_3]]$ is called a $3$-fold commutator, and it is usually denoted  by $[g_1, g_2, g_3]$. Inductively, an $n$-fold commutator is defined as $[g_1, \dots, g_n]=[g_1, [g_2, \dots, g_n]]$. 


We say that $G$ is a \emph{$\tau_2$-group} if it is finitely generated,  torsion-free, and  $2$-step nilpotent.  Observe that, in this case, $G/Z(G)$ and $Z(G)$ are free abelian groups of finite rank. The \emph{rank} of an abelian group $H$ is the minimum cardinality of a  generating set of $H$, which we call \emph{basis} of $H$. 

Let $C=\{c_1, \dots, c_m\}\subset G$ be a basis of a free abelian subgroup  satisfying $G'\leq \langle C\rangle \leq Z(G)$, for $G$ a $\tau_2$-group, and let $\{a_i \langle C \rangle \mid i=1, \dots, n\}$ be a basis of  the abelian group $G/\langle C\rangle$. 
For each $g\in G$ there exist   $\alpha_i(g), \gamma_j(g) \in \mathbb{Z}$ such that
\begin{equation}\label{Malcevrep}
g= a_1^{\alpha_1(g)} \dots a_{n}^{\alpha_{n}(g)} c_1^{\gamma_1(g)} \dots c_{m}^{\gamma_{m}(g)}.
\end{equation}
The expression \eqref{Malcevrep} is called a \emph{Malcev representation} of  $g$ with respect to $(A;C)$. The \emph{Malcev coordinates} of $g$ are given by the tuple $$(\alpha_1(g), \dots, \alpha_{\eta}(g), \gamma_1(g), \dots, \gamma_{\mu}(g)).$$ While the $\gamma_i(g)$'s are  unique,  the $\alpha_i(g)$'a are, in general, unique only up to multiples of the order of $a_i\langle C \rangle$. The following  terminology will be used extensively through the paper:

\begin{dfn}\label{malcevBDfn}
Following the notation above, $(A;C)$ is called a \emph{Malcev basis} of $G$ if $G/\C$ is free abelian, i.e.\ if no $a_i\C$ has finite order, and  thus the Malcev coordinates of all elements from $G$ are unique.
\end{dfn}
%
%
%
%
%
In this case it is useful to regard the symbols $\alpha_i$ and $\gamma_t$ as maps
$
\alpha_i, \gamma_t: G \to \mathbb{Z},
$
sending each $g\in G$ to its Malcev coordinates $\alpha_i(g)$ and $\gamma_j(g)$, respectively.  We remark that the notion of a Malcev basis can be formulated in a  more general setting for  groups of any nilpotency step, possibly with torsion.

Let $F_m=F_m(A)$ be the free group generated by $A=\{a_1, \dots, a_m\}$, and let 
$
T_{j,m}=\{[a_{i_1}, \dots, a_{i_j}] \ | \ 1\leq i_1, \dots, i_j \leq m\}
$
be the set of all $j$-fold commutators on the $a_i$'s. The \emph{free $s$-step rank-$m$ nilpotent group with basic generating set $A$} 
is
$$
N_{s,m}= N_{s,m}(A) = F_m/ \langle \langle T_{s+1, m} \rangle \rangle = \langle a_1, \dots, a_m \ | \ [a_{i_1}, \dots, a_{i_{s+1}}]=1 \ \hbox{for all} \ i_j \rangle.
$$
Here $\langle \langle T_{j,m} \rangle \rangle$ denotes the normal closure of  $T_{j,m}$ in $F_m$. 

The  following  holds in any group: 
$$[xy,z]=y^{-1}[x,z]y[y,z], \quad [x,y]=[y,x]^{-1}, \quad \quad \hbox{for all} \ x,y,z.$$
Using this and the fact that commutators belong to the center of $G$ (for $G$ a $\tau_2$-group), we obtain that, for a fixed $g\in G$, the maps  $y \mapsto [g,y]$ and $x\mapsto [x,g]$ are homomorphisms from $G$ into $[G,G]$. We will implicitly  use this fact from now on.


\subsection{E-definability}

In what follows we   use non-cursive boldface letters such as $\mathbf{a}$ to denote tuples of elements: e.g.\ $\mathbf{a}=(a_1, \dots, a_n)$.
\begin{dfn} 
%
Let $\mathcal{M}=\left(M; f_i, r_j, c_k \mid i,j,k \right)$ be an algebraic structure (for the purposes of this paper, $\mathcal{M}$ is a group or a ring), where $M$ is the universe set of $\mathcal{M}$, and the  $f_i, r_j, c_k$ are the function, relation, and constant symbols of $\mathcal{M}$. A set $A \subset M^m$ is called \emph{definable by equations 
} in $\mathcal{M}$, or \emph{e-definable}, if there exists a finite system of equations over $\mathcal{M}$,
$
\Sigma_A(x_1,\ldots,x_m,y_1, \dots, y_n)$, on variables $\mathbf{x}=(x_1, \dots, x_m) \in M^m$ and $\mathbf{y}=(y_1, \dots, y_n) \in M^n$,
such that, for any tuple $\mathbf{a}\in M^m,$ we have that  $\mathbf{a} \in A$ if and only if $\Sigma_A(\mathbf{a},\mathbf{y})$ has a solution $\mathbf{y}\in M^n$.
\end{dfn}


A function 
$
f:X \subset M^k \to M^l
$ 
is called \emph{e-definable} in $\mathcal{M}$ if its graph 
$
\left\{\left(\mb{a},f(\mb{a})\right)\mid \mb{a}\in X\right\} \subset M^{k+l}
$
is e-definable in $\mathcal{M}$. Similarly, a relation $r:X \to \{0,1\}$ is \emph{e-definable} in $\mathcal{M}$ if its graph
$
\left\{\mb{a} \mid r(\mb{a})=1\right\}
$
is e-definable in $\mathcal{M}$.
\begin{dfn}\label{edfnDfn}
An algebraic structure $\mathcal{A}= \left(A; f, \dots, r, \dots, c, \dots \right)$  is called \emph{e-definable} in another  structure $\mathcal{M}$ if there exists an embedding map 
$
\phi: A \hookrightarrow M^k
$ 
for some $k$, called \emph{defining map}, such that the following holds: 
\begin{enumerate}
\item The image of $\phi$ is e-definable in $\mathcal{M}$. 
\item For every function $f=f(x_1, \dots, x_n)$ of $\mathcal{A}$,
the induced function $\phi(f)$ given by
$
\phi(f)\left(\phi(x_1), \dots, \phi(x_{n})\right)= \phi\left(f(x_1, \dots, x_n)\right)
$
is e-definable in $\mathcal{M}$.
%
%
\item Similarly, for every relation $r$ of $\mathcal{A}$, the induced relation 
$
\phi(r)
$
(with a meaning analogous to that of $\phi(f)$) is e-definable in $\mathcal{M}$. 
\end{enumerate}
It follows that the structure $\phi(\mathcal{A})=\left(\phi(A); \phi(f), \dots, \phi(r), \dots, \phi(c), \dots \right)$ is isomorphic to $\mathcal{A}$. Often, after describing a defining map $\phi$, we identify $\phi(\mc{A})$ with $\mc{A}$. 
%
\end{dfn}
\begin{example}\label{group_interpretations}
%
%
The center $Z(G)$ of a finitely generated group $G =\langle g_1, \dots, g_n \rangle$ is e-definable in $G$ as a set. Indeed,  $x\in G$ belongs to $Z(G)$ if and only if it commutes with all $g_i$'s, and hence $Z(G)$ (seen as a set) is defined in $G$ by means of the following system of equations on the single variable $x$:
$$
\bigwedge_{i=1}^{n}  \left( [x,g_i]=x^{-1}g_i^{-1}xg_i=1 \right).
$$
If, additionally, we regard $Z(G)=\left(Z(G); \cdot, {}^{-1}, 1 \right)$ as an algebraic structure with  operations and constants induced from $G$, then $Z(G)$ is still e-definable in $G$ with defining map ${\it id}: Z(G) \to G$, $id(g)=g$.

For another example, let $G$ be a group with  finite $[x,y]$-width $n$ (see below in this section). Then any $g\in G'$ can be written as a product of exactly $n$ commutators (adding trivial ones if necessary), and thus $G'$ is e-definable in $G$ by means of the equation $x=\left[x_1,y_2\right]\cdots\left[x_n,y_n\right]$.  

%
%
%
%
%
\end{example}
Given a tuple $\mathbf{a}=(a_1, \dots, a_n)\in A^n$ and a map $\phi: A \to M^k$, we denote by $\phi(\mathbf{a})$  the tuple in $M^{nk}$ consisting in the components of $\phi(a_1)$, followed by the components of $\phi(a_2)$, and so on. The following is a fundamental property of e-definability:
\begin{lemma}
\label{RedLemma}
%
If $\mathcal{A}$ is e-definable in $\mathcal{M}$ (with defining map  $\phi: A \to M^k$), then for every system of equations $S(\mathbf{x})=S(x_1, \dots, x_n)$ over $\mathcal{A}$, there exists a system of equations $S^*(\mathbf{y}, \mathbf{z})=S^*(y_1, \dots, y_{kn}, z_1, \dots, z_m)$ over $\mathcal{M}$, such that $\mathbf{a}$ is a solution to $S$ in $\mathcal{A}$ if and only if $S^*(\phi(\mathbf{a}), \mathbf{z})$ has a solution $\mathbf{z}$ in $\mathcal{M}$. Moreover, all solutions $\mathbf{b}, \mathbf{c}$ to $S^*$ arise in this way, i.e.\  $\mathbf{b}=\phi(\mathbf{a})$ for some solution $\mathbf{a}$ to $S$. 
\end{lemma}
\begin{proof}
It suffices to follow step by step the proof of Theorem 5.3.2 from \cite{Hodges}, which states that the above holds when $\mc{A}$ is interpretable by first order formulas in $\mc{M}$. 
\end{proof}
Of course, $S$ has a solution in $\mathcal{A}$ if and only if $S^*$ has a solution in $\mathcal{M}$. Recall that $\mc{D}(\mathcal{A})$ denotes the Diophantine problem over $\mathcal{A}$ (see Definition \ref{DiophDfn}). One immediately obtains:
\begin{cor}\label{RedCor}
If $\mathcal{A}$ is e-definable in $\mathcal{M}$, then $\mc{D}(\mc{A})$ is reducible to $\mc{D}(\mc{M})$. Consequently, if $\mc{D}(\mc{A})$ is undecidable, then so is $\mc{D}(\mc{M})$. 
\end{cor}
For example, due to the negative answer to Hilbert's 10th problem,  $\mc{D}(\mathcal{M})$ is undecidable for any $\mathcal{M}$ in which the ring $\mathbb{Z}$ is e-definable.

\subsection{Ring of scalars of a $\tau_2$-group}\label{ringSc}
Through this paper, by \emph{ring} we mean an associative ring with identity. Below we say that a map $f: M\times M \to N$ between abelian  groups $M$ and $N$ is \emph{bilinear} if, for all $a\in A$, the maps $f(a, \cdot)$ and $f(\cdot, a)$ from $M$ to $N$ are group homomorphisms. 

\begin{dfn}\label{ringScalars}
Let  $f:M \times M \to N$ be a bilinear map between abelian groups. A commutative ring $A$ is called a \emph{ring of scalars}  of $f$  if there exist faithful actions of $A$ on $M$ and $N$ (by endomorphisms), such that $f(r x,y)=f(x,r y)=r f(x,y)$ for all $r\in A,$  $x,y\in M$.
\end{dfn}
The set of endomorphisms of an abelian group $M$ forms a ring once we equip it with the operations of addition and composition (henceforth called multiplication). We denote such ring by $\End(M)$. As done already in the previous definition, we  simply write $\alpha x$ instead of $\alpha(x)$, for $\alpha\in \End(M)$, $x\in M$.

Since the actions of a ring of scalars $A$ on $M$ and $N$ are faithful, there are natural embeddings $A\hookrightarrow \End(M)$ and $A\hookrightarrow \End(N)$. From now on we identify any such $A$ with the corresponding image under the natural embedding into $\End(M)$. We say that $A$ is \emph{maximal} if for any other ring of scalars $B$, we have $B\leq A$. Of course, if it exists, such maximal ring is unique.

Let now $G$ be a non-abelian $\tau_2$-group. Then, as explained previously,  $G/Z(G)$ and $G'$ are non-trivial free abelian groups. Also, the map
\begin{align}
f: G/Z(G) \times G/Z(G) &\rightarrow G' \label{tau2map}\\ 
\left(gZ(G), hZ(G)\right) &\mapsto [g,h] \nonumber
\end{align}
is  well-defined, non-degenerate, bilinear, and its image generates the whole $G'$.  The maximal  ring of scalars of $f$ is also called the maximal \emph{ring of scalars of $G$}. 

%
%
%


%% file: everything_after_intro_and_prel.tex
\section{Defining $\mathbb{Z}$ in $\tau_2$-groups}\label{sectionIntroduction}

\subsection{Small centralizers and maximal ring of  scalars}
In this section we prove that $\mathbb{Z}$ is e-definable in any $\tau_2$-group $G$ that has two non-commuting c-small elements. Additionally, we show that if  certain  elements of $G$ are c-small, and they pairwise do not commute, then  the maximal ring of scalars of $G$ is $\mathbb{Z}$. We end the section by proving that, in this case, $G$ cannot be decomposed into a direct product of  non-abelian factors. 
%

\begin{dfn}\label{generalPosDfn}
An element $g$ of a group $G$ is  \emph{centralizer-small} (or just \emph{c-small}) if  
$C(g)=\{ g^t z \ | \ t\in \mbb{Z}, \ z \in Z(G)\},$ where $C(g)$, the centralizer of $g$, denotes the set of elements in $G$ that commute with $g$. 
\end{dfn}

%
%
%
%
%
%

As mentioned in the introduction, the  techniques used to prove the following result have a resemblance with some arguments from Duchin, Liang, and Shapiro in \cite{Duchin}, and from Romankov in \cite{Romankov}.


\begin{thm}\label{mainThmEq2}
Let $G$ be a $\tau_2$-group with Malcev basis $(A;C)$. Suppose $G$ has   two   non-commuting c-small elements $a,b$. Then $\mbb{Z}$  is e-definable in $G$. In particular, $\mc{D}(G)$ is undecidable. 
\end{thm} 
\begin{proof}
Consider the set $Z=[a, C_{G}(b)]=\{[a, x] \ | \ x \in C_{G}(b)\}$, and denote $c=[a,b] \neq 1$. Using that $b$ is c-small, $Z=\{c^t \ | \ t \in \mbb{Z}\}$. 
 Since $G$ is torsion free, the map sending $c^t \in Z$ to  $t\in Z$ is a bijection. 
Torsion-freeness also allows us to well-define   binary and unary operations  $\oplus$, $\ominus$, $\odot$ in $Z$ by letting $$c^{t_1} \oplus c^{t_2} = c^{t_1+t_2}, \quad \ominus(c^t)= c^{-t}, \quad \hbox{and} \quad c^{t_1} \odot c^{t_2} =c^{t_1t_2}.$$ We are going to prove that the ring $(Z, \oplus, \odot, c^0, c^1)$, which is isomorphic to $\mathbb{Z}$, is e-definable in $G$. 

An element $g\in G$ belongs to $Z$ if and only if the following identities hold for some $y\in G$: $g=[a, y], \ [y,b]=1$. In other words, $g\in Z$ if and only if $g$ is part of a solution to the system of  equations $\left( x=[a, y] \right) \wedge \left([y,b]=1\right)$ on variables $x, y$. Hence, $Z$ is e-definable in $G$ as a set. Now let $g_1, g_2, g_3 \in Z$. Clearly, $g_1 \oplus g_2 = g_3$ if and only if $g_1g_2=g_3$. It follows that the graph of $\oplus$ is e-definable in $G$: it suffices to take the system formed by the equation $xy=z$ together with equations that ensure  $x,y,z\in Z$. 
Analogously, and taking the equation $xy=1$ instead of $xy=z$, one sees that  $\ominus$ is e-definable in $G$.

Regarding $\odot$, consider the following system over $G$ on variables $x_i$, $i=1,2,3$, and $x_i'$, $i=1,2$.
\begin{equation}\label{odot}
\begin{cases}
x_1=[x_1', b],  
&[x_1', a]=1,\\
x_2=[a, x_2'], 
&[x_2',b]=1,\\
x_3=[x_1', x_2'].
\end{cases}
\end{equation}
Suppose $x_1, x_2, x_3, x_1', x_2'$ is a solution to \eqref{odot}. Since $a$ and $b$ are c-small, $x_1' = a^{t_1} z_1$ and $x_2' = b^{t_2} z_2$ for some $t_i \in \mbb{Z}$ and some $z_i \in Z(G)$, $i=1,2$. Moreover, $x_1= c^{t_1}$ and $x_2=c^{t_2}$. We also have $[x_1', x_2'] = c^{t_1 t_2} =x_3,$ and hence $x_3= c^{t_1t_2}= x_1 \odot x_2$. Conversely, let $x_1, x_2, x_3$ be three elements from $G$ such that $x_1 \odot x_2 = x_3$. Then it is easy to verify that there exist $x_1', x_2'$ such that $x_1, x_2, x_3, x_1', x_2'$ form a solution to \eqref{odot}. We conclude that $x_1 \odot x_2 = x_3$ if and only if $x_1, x_2, x_3$ are part of a solution to \eqref{odot}. Similarly as before, $\odot$ is e-definable in $G$.

This completes the proof, since the ring $(Z; \oplus, \ominus, \odot, c^0, c^1)$ is  e-definable in $G$ and it isomorphic to the ring of integers $(\mathbb{Z}; +, -, \cdot, 0, 1)$. 
%
\end{proof}

%
%
%
%
%
%
%
%
%
%
%
%
%
%
%
%
%
%
%

In \cite{Part_3} we explain how to extend Theorem \ref{mainThmEq2} to the class of finitely generated nilpotent groups $G$ (of any nilpotency step, possibly with torsion). This allows one to prove, for example, that $\mc{D}(N)$ is undecidable over any non-abelian free nilpotent group, recovering one of the results from \cite{Duchin}.

\begin{thm}\label{CSmallMaxZ}
%
Let $(A;C) = (a_1, \dots, a_{n}; c_1, \dots, c_m)$ be a Malcev basis of a $\tau_2$-group $G$ ($n\geq 2$). Assume  that  $[a_i, a_j]\neq 1$ for all $i\neq j$, and  that $a_i$ is c-small for all $i=1, \dots, n$. Then the maximal ring of scalars  of $G$ is isomorphic to the ring of integers $\mathbb{Z}$.
%
\end{thm}
\begin{proof}
Let $R$ be a ring of scalars of $G$, and fix an element $r\in R$. For each $i=1, \dots, n$, choose a representative $b_{r,i}$  of $r\left(a_iZ(G)\right)$, so that $b_{r,i}Z(G)=r\left(a_iZ(G)\right)$. Then, in $G$:
$
[a_i, b_{r,i}]= [a_iZ(G), b_{r,i}Z(G)] = [a_i Z(G), r\left(a_iZ(G)\right)]= r[a_iZ(G), a_iZ(G)]=r[a_i, a_i]=1,
$
and hence $a_i$ and $b_{r,i}$ commute. Since $a_i$ is c-small, $$b_{r,i}=a_i^{t_{r, i}} z_i$$ for some $t_{r,i}\in \mathbb{Z}$ and $z_i \in  Z(G)$. Thus
$$
r\left(a_iZ(G)\right) = a_i^{t_{r,i}}Z(G).
$$
Now, for any $i,j$:
\begin{align}
[a_i, a_j]^{t_{r,i}} & = [a_i^{t_{r,i}}, a_j] =  [r\left(a_{i}Z(G)\right), a_jZ(G)]=\nonumber\\ & = [a_{i}Z(G), r\left(a_{j}Z(G)\right)] = [a_i, a_j]^{t_{r,j}}.\nonumber
\end{align}
Since $G$ is torsion-free and $[a_i, a_j]\neq 1$,  we obtain $t_{r,i}= t_{r,j}$ for all $i\neq j$. It follows that, for all $r\in R$, there exists $t_r \in \mathbb{Z}$ such that $r(a_iZ(G))= a_i^{t_r}Z(G)$ for all $i$. By definition, $\C\leq Z(G)$, and $G/\C$ has basis $A/\C$. Hence, $G/Z(G) $ is generated by  the $a_i Z(G)$'s. It follows that $r(gZ(G))= g^{t_r}Z(G)$ for all $g\in G$.

Since $G/Z(G)$ is a free abelian group, the map $\phi: R \to \mathbb{Z}$ given by $\phi(r)=t_r$ is a ring homomorphism. 
Moreover, $\phi$ is exhaustive, because given $k\in \mathbb{Z}$, we have $$\phi\left(\sum_{i=1}^k id \right) = \sum_{i=1}^ k \phi(id) = k,$$ where $id$ denotes the identity endomorphism of $G/Z(G)$, i.e.\ the identity element of $R$. Finally, notice that if $\phi(t)=t_r=0$ for some $t$, then $r(gZ(G))= Z(G)$ for all $g\in G$. Hence,  $r$ is the $0$ element of $R$. We conclude that $\phi$ is a ring isomorphism. This proves that any ring of scalars of $G$ (with an identity) is isomorphic to $\mathbb{Z}$. In particular, this is true of the maximal ring of scalars of $G$.
\end{proof}
We remark that Theorem \ref{CSmallMaxZ}  is still true as long as the complement of the commutativity graph between the $a_i$'s is connected.
%
\begin{prop}\label{dirIndec}
Suppose $\mathbb{Z}$ is the maximal ring of scalars  of a $\tau_2$-group $G$. Then $G$ cannot be decomposed into a direct product of  non-abelian subgroups. 
\end{prop}
\begin{proof}
Suppose $G= H \times K$ for some  non-abelian subgroups $H,K$ of $G$. Then $H$ and $K$ are non-abelian $\tau_2$-groups, $Z(G)=Z(H)\times Z(K)$, and $G/Z(G)$ decomposes non-trivially due to the isomorphism $G/Z(G) \cong  H/Z(H) \times K/Z(K)$. Moreover, using that $[(h_1,k_1),(h_2,k_2)]=\left([h_1,h_2], [k_1,k_2]\right)$ for all $h_1,h_2\in H, k_1,k_2\in K$, we obtain  $G'=H' \times K'$.  

Consider the natural actions of $\mathbb{Z}^2$ on $H/Z(H) \times K/Z(K)$ and on $H' \times K'$ defined by component-wise exponentiation (or component-wise multiplication if one is using additive notation): 
$$
(r_1, r_2)(h, k)= (h^{r_1}, k^{r_2})
$$ 
for all $(r_1, r_2)\in \mathbb{Z}^2$ and all $(h, k)$ in $ H/Z(H) \times K/Z(K)$ or in $H' \times K'$.
Since $G$ is a $\tau_2$-group, these are faithful actions by endomorphisms that satisfy $[\mb{r}\mb{u},\mb{v}]=[\mb{u},\mb{r}\mb{v}]=\mb{r}[\mb{u}, \mb{v}]$ for all $\mb{r}\in \mathbb{Z}^2$ and all $\mb{u},\mb{v}\in H/Z(H) \times K/Z(K)$. 
%
%
Thus, $\mathbb{Z}^2$ is a ring of scalars of $G=H \times K$. By  definition, $\mathbb{Z}^2$ is a subring of  the maximal ring of scalars of $G$, which is $\mathbb{Z}$ by hypothesis - a contradiction.  
%
%
%
%
 %
\end{proof}


%
%
%
%
%
%
%
%
%
%
%

\subsection{C-small elements and Malcev bases}\label{CSmallMalcev}
Through this section we let $G$ be a $\tau_2$-group with Malcev basis $(A; C)=(a_1, \dots, a_n; c_1, \dots, c_m)$, for some $n\geq 2$ and $m\geq 1$. By definition, $G' \leq \langle C\rangle \leq Z(G)$. Hence, for each $[a_i, a_j]$ there exist unique integers $\lambda_{t}^{ij}$ such that:
\begin{equation}\label{lambdas}
[a_i, a_j] = \prod_{t=1}^{m} c_t^{\lambda_t^{ij}}.
\end{equation}
For notational convenience, we will sometimes write $\l_{t,i,j}$ instead of $\l_{t}^{ij}$. As explained in Subsection \ref{nilpSub}, there exist maps $\alpha_i : G \to \mathbb{Z}$, $\gamma_j : G \to \mathbb{Z}$, called Malcev coordinates, such that any $x\in G$ can be written uniquely as:
\begin{equation}\label{coord}
x= \prod_{i=1}^n a_i^{\alpha_i(x)} c(x), \quad \hbox{where} \quad c(x) = \prod_{j=1}^m c_j^{\gamma_j(x)}.
\end{equation}
%
%
%
In what follows we use the  Malcev basis $(A;C)$ to construct  homogeneous systems of linear Diophantine equations $S_{\ell}$  such that, if $S_{\ell}$ admits only the trivial solution, then $a_{\ell}$ is c-small, provided another minor condition is met  ($\ell=1, \dots, n$).  We start by expressing the Malcev coordinates of a commutator $[x,y]$ in terms of the Malcev coordinates $\alpha_i(x)$, $\alpha_i(y)$ of $x$ and $y$. 
\begin{lemma}\label{commxycoord}
%
%
%
%
The following identity holds for any $x, y\in G$:
\begin{equation}
[x,y]= \prod_{t=1}^{m} c_{t}^{\sum_{i,j=1}^{n} \lambda_{t}^{ij} \alpha_i(x) \alpha_j(y) }.
\end{equation}
\end{lemma}

\begin{proof}
We use that $G'\leq Z(G)$, and that the map $[\cdot, \cdot]$ behaves "bi-linearly", together with equations \eqref{coord} and \eqref{lambdas}.
\begin{align}
&[x , y] = [\prod a_i^{\alpha_i(x)} c(x), \prod a_i^{\alpha_i(y)} c(y)] = [\prod a_i^{\alpha_i(x)}, \prod a_i^{\alpha_i(y)}] =\nonumber \\ 
= &\prod_{i, j=1}^{n}[a_i, a_j]^{\alpha_i(x) \alpha_j(y)}= \prod_{i, j=1}^{n}\prod_{t=1}^{m} c_{t}^{\lambda_{t}^{ij} \alpha_i(x) \alpha_j(y)} =\prod_{t=1}^{m} c_{t}^{\sum_{i,j=1}^{n} \lambda_{t}^{ij} \alpha_i(x) \alpha_j(y)}. \nonumber
\end{align}
\end{proof}
This result allows one to express the equation $[x, y] = w$ in $G$ as a system of Diophantine equations. More precisely, if $x, y \in G$ and $w \in G' \leq \langle C \rangle$ (in which case $\alpha_i(w) =0$ for all $i$), 
then $[x, y]=w$ if and only if
\begin{equation}\label{mainsystem}
\sum_{i,j=1}^{n} \lambda_{t}^{ij} \alpha_i(x) \alpha_j(y) = \gamma_{t}(w) \quad \hbox{for all} \quad t=1, \dots, m.
\end{equation}
%
%
%
%
%
Hence, one can find all triples $x,y,w \in G$ such that $[x, y] =w$ by finding all solutions $\alpha_i(x), \alpha_j(y), \gamma_k(w)$ to the system formed by the equations of \eqref{mainsystem}, with the  $\alpha_i(x), \alpha_j(y), \gamma_k(w)$'s seen as variables, and then taking $$x = \prod a_i^{\alpha_i(x)} c_x, \quad y=\prod a_j^{\alpha_j(y)}c_y, \quad \hbox{and} \quad z=\prod c_t^{\gamma_k(w)},$$ for arbitrary $c_x, c_y \in \langle C \rangle$. One may also take $x$, $y$, or $w$ to be constants. In the next result we apply this strategy to the equation  $[a_{\ell}, x] =1$. 


\begin{cor}\label{centralizergenerators}
For any $a_{\ell}$ and any $x\in G$, we have that $[a_{\ell}, x]=1$ if and only if  
\begin{equation}\label{Asystem} \begin{pmatrix} -\lambda_1^{1, \ell} & \dots & -\lambda_{1}^{\ell-1, \ell} & \lambda_1^{\ell, \ell+1} & \dots & \lambda_1^{\ell, n} \\ \vdots & \ddots & \vdots & \vdots & \ddots & \vdots \\ -\lambda_{m}^{1,\ell} & \dots & -\lambda_{m}^{\ell-1, \ell} & \lambda_m^{\ell, \ell+1} & \dots & \lambda_m^{\ell, n}  \end{pmatrix} \left( \begin{array}{c} \alpha_1(x) \\ \vdots \\ \alpha_{\ell-1}(x) \\ \alpha_{\ell+1}(x) \\ \vdots \\ \alpha_{n}(x) \end{array} \right)
= \left( \begin{array}{c} 0 \\ \vdots \\ 0 \end{array} \right) \end{equation}
%
%
%
\end{cor}
\begin{proof}
Note that $\alpha_i(a_{\ell})=0$ if $i\neq \ell$ and  $\alpha_{\ell}(a_{\ell})=1$ otherwise.  Applying   \eqref{mainsystem} on $[a_{\ell}, x]=1$, we obtain
\begin{equation}\nonumber
\sum_{j=1}^{n} \lambda_{t}^{\ell j} \alpha_j(y) = \gamma_t(1) \quad \hbox{for all} \quad t=1, \dots, m.
\end{equation}
The result now follows from the identities $\l_{t}^{ij}=-\l_{t}^{ji}$, $\gamma_t(1)=0$, and $\l_{t}^{\ell, \ell}=0$, which hold for all $t$, $i$, $j$. Since $\l_{t,\ell, \ell}=0$ for all $t$, in  \eqref{Asystem} we omit the column  that corresponds to these $\l$'s, and consequently we remove the variable $\alpha_{\ell}(x)$ as well.
\end{proof}
We look at \eqref{Asystem} as a system of linear Diophantine equations on variables $\left\{ \alpha_i(x) \mid i \neq \ell \right\}$, and we denote its matrix by $M_{\ell}$.  Below we give the main result of this section.  It provides sufficient conditions for   $a_{\ell}$ being c-small, and for having $Z(G)=\C$. These conditions are formulated in terms of the \emph{rank} of $M_{\ell}$, which is the maximum number of $\mathbb{Z}$-linearly  independent rows (or columns) of $M_{\ell}$.  Equivalently, it is its maximum number of $\mathbb{R}$-linearly independent rows or columns  (this can be seen by computing the Smith normal form of $M$). A set of vectors $\mathbf{v}_1, \dots, \mathbf{v}_t$ is \emph{$K$-linearly independent} if the equation $k_1 \mathbf{v}_1 + \dots + k_{t}\mathbf{v}_t=0$, $k_i\in K$, has only the solution $k_i=0$ for all $i$. 

\begin{prop}\label{ThetaCor}
%
%
Suppose $\rank(M_{\ell})=n-1$ for some $\ell=1, \dots, n$. Then $[a_{\ell}, x]=1$ if and only if \begin{equation}\label{eq100}x=a_{\ell}^{\alpha_{\ell}(x)} c(x).\end{equation} If, additionally, $\l_{t}^{\ell, k}\neq 0$ for some $t$ and $k\neq \ell$, then $Z(G)=\C$, and $a_{\ell}$ is c-small.
\end{prop}
\begin{proof}
If $\rank(M_{\ell})=n-1$, then \eqref{Asystem} admits only the trivial solution. Then, by Corollary \ref{centralizergenerators}, $[a_{\ell},x]=1$ if and only if $\alpha_{i}(x)=0$ for all $i\neq \ell$. It follows that $[a_{\ell},x]=1$ if and only if the Malcev representation  \eqref{coord} of $x$ is as in \eqref{eq100}. Assume now, additionally, that $\l_{t,\ell, k}\neq 0$ for some  $t$ and $k$, and let   $g\in Z(G)$. Then $[a_{\ell},g]=1$, and by what we just proved, $$g=a_{\ell}^{\alpha_{\ell}(g)} c(g).$$ Applying  Corollary \ref{centralizergenerators}  to $a_k$ and $g$, we obtain that $\l_{t',k, \ell}\alpha_{\ell}(g)=0$ for all $t$'. Since  $\l_{t,k,\ell}\neq 0$, we have $\alpha_{\ell}(g)=0$, and thus $g=c(g)\in \C$. We conclude that $Z(G)=\C$. That $a_{\ell}$ is c-small follows now from \eqref{eq100} and $Z(G)=\C$.
\end{proof}


%
%


%




\section{Random  $\tau_2$-groups}\label{randomSec}



\subsection{$\tau_2$-presentations and their span} 
%

%
In this subsection we study groups $G$  that admit the following specific type of polycyclic presentation:
\begin{equation}\label{NPpres}
G= \langle A, C \mid [a_i, a_j]=\prod_{t=1}^m c_t^{\l_t^{ij}},\ [A,C]=[C,C]=1, \ 1\leq i < j \leq n \rangle,
\end{equation}  
for some  sets $A= \{a_1, \dots, a_n\}$ and $C=\{c_1, \dots, c_m\}$, with $n,m\geq 0$.  We call \eqref{NPpres} a \emph{$\tau_2$-presentation}, and we let $\mc{P}(n,m)$ denote the set of all such presentations. More precisely, $P \in \mc{P}(n,m)$ if and only if $P$ has the form \eqref{NPpres} for some exponents $\lambda_t^{ij}$, and some sets $A$, $C$ with $|A|=n$, $|C|=m$. We remark that, through the paper (and as already done in \eqref{NPpres}), we often make no distinction between a presentation and its corresponding group.  

The goal of this subsection is to prove the following:
\begin{thm}\label{SpanThm}
Let $G$ be a group, and let $n,m$  be two  integers greater or equal than $0$. Then the following two statements are equivalent. 1)  $G$ admits a  $\tau_2$-presentation  from $\mc{P}(n,m)$.
2) $G$ is a $\tau_2$-group  with
 $\rank(G/Z(G)) + \rank(Z(G)) = n + m$, and
 $\rank(G') \leq m \leq \rank(Z(G))$.

Consequently, $G$ is a $\tau_2$-group if and only if it admits a $\tau_2$-presentation.
\end{thm}
%
%

While proving this we will obtain interesting information regarding the structure of $G$, see Corollary \ref{superCor} and Lemma \ref{neatFormGLemma}. In particular, we will show that if $G$ has presentation \eqref{NPpres}, then $(A;C)$ is a Malcev basis of $G$. 

\begin{lemma}\label{degenerate_cases}
Theorem \ref{SpanThm} holds if $n=0,1$ or $m=0$.
\end{lemma}
\begin{proof}
The forward direction is immediate. To prove the converse, assume that $G$ is a $\tau_2$ group satisfying Items 1) and 2) in Theorem \ref{SpanThm}. If $m=0$, then $\rank(G')\leq m=0$. Hence, $G$ is a free abelian group.  Moreover, $\rank(G/Z(G))+\rank(Z(G))=\rank(G)=n$, and thus $G$ admits the $\tau_2$-presentation $$\langle a_1, \dots, a_n \mid [a_i, a_j]=1 \ \hbox{for all} \ i,j \rangle.$$ This presentation belongs to $\mc{P}(n,0)$.

Suppose $n=0,1$. By hypothesis, $\rank(Z(G))=n+m-\rank(G/Z(G))$, and $m \leq n+m -\rank(G/Z(G))$. Consequently, $\rank(G/Z(G))\leq n\leq 1$. This implies, again, that $G$ is a free abelian group. Then, as before, $\rank(G/Z(G)) + \rank(Z(G))= \rank(G)=n+m$, and $G$ admits the $\tau_2$-presentation $\langle a_1, c_1, \dots, c_m \mid [a_1, c_i]=[c_j, c_k]=1 \ \hbox{for all} \ i,j,k \rangle$ (if $n=0$ omit  $a_1$). This presentation belongs to $\mc{P}(n,m)$, as needed.
\end{proof}


We adopt the following standard convention regarding the projection of elements onto quotient groups: Suppose  $K$ has been obtained from another group (or monoid) $H$ by adding some relations, and let $\pi:H\to K$ be the canonical projection of $H$ onto $K$. To avoid referring to $\pi$, we  speak of  \emph{elements $h$ from $H$ seen in  $K$} (or \emph{projected onto $K$}), rather than of elements $\pi(h)$. Similarly, for $h_1, h_2\in H$, instead of  writing $\pi(h_1) = \pi(h_2)$, we  say that  \emph{$h_1=h_2$ in $K$}, or $h_1 =_K h_2$. For example, if a group $G$ has a generating set $A$,  then we  speak of   \emph{($A^{\pm 1}$)-words seen in $G$}, and of  \emph{equalities between ($A^{\pm 1}$)-words taking place in $G$.}  
 
 The next result is fundamental for our purposes.
%
%
%
\begin{lemma}\label{malcevBa}
Let $G$ have a $\tau_2$-presentation as in \eqref{NPpres}, for $n\geq 2$, $m\geq 1$. Then there exist maps $\alpha_i, \gamma_t: G \to \mathbb{Z}$ such that each $g\in G$ can be written uniquely as
\begin{equation}\label{eq}
g= \prod_{i=1}^n a_i^{\alpha_i(g)} \prod_{t=1}^m c_t^{\gamma_t(g)},
\end{equation}
in the sense that if $g= \prod_i a_i^{x_i} \prod_t c_t^{y_t}$ for some  $x_i, y_t$, then, necessarily, $x_i = \alpha_i(g)$ and $y_t=\gamma_t(g)$ for all $i, t$.
\end{lemma}
\begin{proof}
Let $N_0 = N(A) \times \mathbb{Z}^m(C)$ be the direct product of the $2$-step free nilpotent group $N(A)$ with basic generating set $A$, and the free abelian group $\mathbb{Z}^m(C)$ with basis $C$. $N_0$ admits the presentation $\langle A, C \mid [A,[A, A]]=[A,C]=[C, C]=1 \rangle$. Hence, $G$ can be obtained from $N_0$ by adding the relations $r_{i,j}=[a_i, a_j]\prod_t c_t^{-\l_{t,i,j}}=1$ for all $1\leq i<j \leq n$. 

Using presentation \eqref{NPpres} and the identity $a_j a_i = a_i a_j [a_i, a_j]^{-1}$, one sees that any $g\in G$ can be written as in \eqref{eq} for some  $\alpha_i$'s and $\gamma_t$'s. We next show that these $\alpha_i$ and $\gamma_t$ are unique. To do so, it suffices to prove that  $g=1$ in  $G$ if and only if $\alpha_i = \gamma_t = 0$ for all $i$ and $t$. 
Assume  that $g=\prod a_i^{\alpha_i} \prod c_t^{\gamma_t}=1$ in $G$.
 %
Then, in $N_0$, $g$ equals a product of conjugates of elements $r_{i, j}$:
\begin{equation}\label{eq5}
g= \prod_i^n a_i^{\alpha_i} \prod_{t=1}^m c_t^{\gamma_t} =_{N_0} \prod_{\epsilon} w_{\epsilon}^{-1} r_{i_{\epsilon}, j_{\epsilon}}^{\pm 1} w_e =_{N_0} \prod_{1\leq i < j \leq n} \left(  [a_i, a_j]^{\ell_{i,j}} \prod_t c_t^{-\ell_{i,j} \lambda_{t, i,j}} \right),
\end{equation}
where $\ell_{i,j}$ is the sum of the exponents of  the $r_{i, j}$ appearing in the the expression above, for each $i<j$. In the last step we have used that all commutators and  $c$'s belong to the center of $N_0$. Hence, the equality \eqref{eq5} takes place in $N_0 = N(A) \times \mathbb{Z}^m(C)$. By definition of direct product, and since $(A; \{[a_i, a_j] \mid i< j\})$ is a Malcev basis of $N(A)$, we have that $\alpha_i = \ell_{j,k}=0$ for all $i$ and $j<k$. We conclude that $\prod_t c_t^{\gamma_t}= 1$ in $N_0$, which can only occur if $\gamma_t = 0$ for all $t$.
\end{proof}
\begin{cor}\label{superCor}
Suppose $G$ is a group with $\tau_2$-presentation \eqref{NPpres}, with $n\geq 2$, $m\geq 1$. Then the following holds: 1) $G$ is a $\tau_2$-group with Malcev basis $(A; C)$.   2) $Z(G)$ admits a basis of the form $C\cup D$ for some set $D$. 3)  $\rank(G/Z(G))=n-|D|$. 4) $\rank(G/Z(G)) + \rank(Z(G)) = n + m$. 5)   $\rank(G') \leq m \leq \rank(Z(G))$.
\end{cor}

\begin{proof}
To see that $G$ is $2$-nilpotent it suffices to show that all commutators belong to the center. To do so, one may use the formulas $[xz,y]=z^{-1}[x,y]z[z,y]$, $[x,y]=[y,x]^{-1}$, and the fact that $G=\langle A\cup C \rangle$, $C\subset Z(G)$, to express any commutator $[g,h]$ as a product of commutators from $\{[a_i, a_j] \mid 1\leq i,j\leq n\}$. By definition, the latter belong to $Z(G)$, and hence so does $[g,h]$. We also obtain that $G' \leq \langle C \rangle$, because each $[a_i, a_j]$ equals a product of elements from $C^{\pm 1}$. 

Let $\alpha_i, \gamma_t: G \to \mathbb{Z}$  be the maps of Lemma \ref{malcevBa}. To prove that $G$ is torsion-free, assume $g^k = 1$ in $G$ for some $g \in G$ and some $k\neq 0$. By the previous Lemma \ref{malcevBa},  we must have have  $\alpha_i(g^k) =0$ for all $i$. Observe that $\alpha_i(g^k)= k \alpha_i(g)$, and so $\alpha_i(g) = 0$ for all $i$. It follows that $g \in \langle C \rangle$, which, again by Lemma \ref{malcevBa}, is a free abelian group with basis $C$, and thus we must have $g=1$. This proves that $G$ is a $\tau_2$-group.

Since $G' \leq \langle C\rangle$, $G/\langle C \rangle$ is an abelian group. By Lemma \ref{malcevBa}, it is free abelian with basis (the projection of) $A$. It follows that $(A;C)$ is a Malcev basis of $G$. 

We next prove that $Z(G)$ has a basis of the form $C \cup D$ for some set $D$. Indeed, it is well known that this occurs if and only if, for any basis $E=\{e_i\mid i\}$ of $Z(G)$, the coordinates of each $c_j\in C$ with respect to $E$ are coprime. To show this is the case, suppose the coordinates of $c_j \in C$ are all divisible by some $d\neq 0$. Then, in $G$,
$$
c_j= \prod_i e_i^{dk_i} = \left(\prod_ie_i^{k_i}\right){}^{d}
$$
for some $k_i$'s (the second equality holds because  $E \subset Z(G)$). Denote $w=\prod e_i^{k_i}$. As before, $\alpha_i(c_j)=0=\alpha_i(w^d)=d\alpha_i(w)$, and so $\alpha_i(w)=0$ for all $i$. It follows that $\gamma_t(w^d) = d\gamma_t(w)$. We finally  obtain that $1=\gamma_j(c_j)=\gamma_j(w^d)= d \gamma_j(w)$. Hence $d=-1, 1$, as needed.

Let $C\cup D$ be a basis of $Z(G)$, with $D=\{d_1, \dots, d_{|D|}\}$. 
We now show that the vectors $\left\{\mathbf{a}_j=(\alpha_1(d_j), \dots, \alpha_{n}(d_j)) \mid j=1, \dots, |D|\right\}$ are linearly independent. Indeed, suppose not. Then there exist integers $k_1, \dots, k_{|D|}$ such that $\sum_j k_j \mathbf{a}_j =\mathbf{0}$. Let $g= \prod_j \scriptstyle{d_j{}^{\scriptscriptstyle{k}_j}}$. Then $\alpha_i(g)=0$ for all $i$. Writing $g$ in its Malcev representation form we obtain, in $G$, $$g=\prod_j d_j^{k_j} = \prod_i a_i^{\alpha_i(g)} \prod_t c^{\gamma_t(g)} = \prod_t c^{\gamma_t(g)}.$$ This is a linear combination in $Z(G)$ between elements of the basis $C \cup D$. Hence, we must have $k_j=\gamma_t(g)=0$ for all $j$ and $t$. This proves that the $\mathbf{a}_j$'s form a set of linearly independent vectors.

From this it follows that $G/Z(G)$, which is equal to $$\langle A, C \mid C= [A,A]= d_1 = \dots = d_{|D|}=1\rangle$$ is a free abelian group of rank $n - |D|$. 
On the other hand, $Z(G)$ has rank $|C|+|D| = m+|D|$, and so $\rank(G/Z(G)) +  \rank(Z(G)) = n + m$.
Finally, since $\langle C \rangle$ is a free abelian group with basis $C$, and $G'\leq \langle C \ranlge \leq Z(G)$,  we obtain $\rank(G')\leq |C|=m \leq \rank(Z(G))$.
%
%
%
\end{proof}

We next show that every  $\tau_2$-group admits a $\tau_2$-presentation:

\begin{lemma}\label{neatFormGLemma}

Any $\tau_2$-group $G$ admits a presentation of the form \begin{equation}\label{eq00}G=\langle A, C\mid  [a_i, a_j]= \prod_{t=1}^{|C|} c_t^{\l_{t}^{ij}}, \ [A,C]=[C,C]=1, \ (1\leq i<j \leq |A|)\rangle,\end{equation}  where $C=\{c_j\}$ is a basis of $Z(G)$ and $A=\{a_i\}$ projects to a basis of $G/Z(G)$. Moreover, the tuple $(A;C)$ is a Malcev basis of $G$. 
\end{lemma}

\begin{proof}
The result is immediate if $G$ is a free abelian group. Assume it is not. Since finitely generated nilpotent groups are finitely presented, we have that $G=\langle A, C \mid r_1, \dots, r_k\rangle$ for some relations $r_i$ (which we see as words on $(A\cup C)^{\pm 1}$). Each $[a_i, a_j]$ can be written as a linear combination of the $c_t$'s, because $G' \leq Z(G)$. Let $\lambda_{t, i,j}$ be the coefficients of these combinations, and let $P \in \mc{P}(n, m)$ be the $\tau_2$-presentation given by $A, C,$ and these $\lambda_{t,i,j}$'s. We claim that $G$ admits the presentation $P$. Indeed, denote by $H$  the group presented by $P$.  It is clear that all relations from $P$ hold in $G$. On the other hand, using the relations from $P$, we may rewrite each word $r_k$, \emph{both in $G$ and $H$,} into the form
$$
r_k=\prod_i a_i^{\alpha_i} \prod_t c_t^{\gamma_t}.
$$
Since $(A; C)$ is a Malcev basis of $G$, and $r_k=1$ in $G$, we must have $\alpha_i=\gamma_t=0$ for all $i,t$, and hence $r_k=1$ in $H$ as well. It follows that $G= H$, as needed. The last statement of the Lemma follows directly from Corollary \ref{superCor}.
\end{proof}

An immediate consequence of the previous two results (Corollary \ref{superCor} and Lemma \ref{neatFormGLemma}) is that:
\begin{cor}
The class of $\tau_2$-groups coincides with the family of groups that admit a $\tau_2$-presentation. 
\end{cor}
%
%
%
%

\begin{proof}[\emph{\textbf{Proof of Theorem \ref{SpanThm}}}]
The degenerate cases $n=0,1$ or $m=0$ were treated in Lemma \ref{degenerate_cases}. Assume $n\geq 2$ and $m\geq 1$. Then the implication $1) \Rightarrow 2)$ follows from Corollary \ref{superCor}. Conversely, suppose $G$ is a $\tau_2$-group satisfying  2), and let $A=\{a_1, \ldots, a_{n(G)}\}$, $C=\{c_1, \ldots, c_{m(G)}\}$  be bases of $G/Z(G)$ and $Z(G)$, respectively.   
By Lemma \ref{neatFormGLemma}, $G$ admits a $\tau_2$-presentation $P \in \mc{P}(n(G), m(G))$ of the form
\begin{equation}\label{eq0}
G=\langle A, C \mid  [a_i, a_j] = \prod_t^{m(G)} c_t^{\lambda_t^{ij}}, \ [A, C]=[C,C]=1, \ 1\leq i< j \leq n(G)\rangle,
\end{equation}
%
%
and $(A;C)$ is a Malcev basis of $G$. Notice that the first relations in \eqref{eq0} express each element of the generating set $\{[a_i,a_j]\mid i< j\}$ of $G'\leq\langle C \rangle\leq Z(G)$ as a linear combination of the $c_t$'s.  It is well known (from basic facts of the theory of abelian groups) that the rank of $G'$ coincides with the rank of the matrix $M$ whose rows are  $\{(\lambda_{1, i,j}, \dots, \lambda_{m(G), i, j})\mid i<j\}$.  Moreover, column operations in $M$ (interchanging two columns, and adding or subtracting one column to another different column)  correspond to changing the basis $C$  (on the other hand, row operations correspond to changing the generating set of $G'$, but this is of no use to us). Hence, by choosing a new basis of $Z(G)$, which we still denote $C$, we may assume that the last $m(G) - \rank(G')$ columns of $M$ are $0$. We thus have
\begin{equation*}\label{eq1}
G= \langle A, C \mid [a_i, a_j]= c_1^{\lambda_{1}^{ij}} \dots  c_{r}^{\lambda_{r}^{ij}} c_{r+1}^0 \dots c_m^0 \dots c_{m(G)}^0,  \ [A,C]=[C,C]=1 \rangle,
\end{equation*}
%
%
%
where $r= \rank(G')$. 
By hypothesis, $r\leq m \leq m(G)$, which allows us to split $C$ into two  sets: $C_0=\{c_1, \dots, c_{m}\}$ and $C_1=\{c_{m+1}, \dots, c_{m(G)}\}$ ($C_1$ is empty if $m(G)=m$), and to rewrite   \eqref{eq1} as
\begin{align}
G= \langle A\cup C_{1}, C_0 \mid \  &[a_i, a_j] =  c_1^{\lambda_{1}^{ij}} \dots  c_{r}^{\lambda_{r}^{ij}} c_{r+1}^0 \dots c_{m}^0, \quad (1 \leq i < j \leq n(G)), \nonumber  \\
&[A, C_{1}]= [C_{1}, C_{1}] = \prod_{t=1}^{m} c_t^0,\ [A\cup C_{1}, C_0 ]= [C_0, C_0 ]=1 \rangle. \nonumber
\end{align}
Finally, notice that  $|A\cup C_{1}|= n(G) + m(G) - m = n+m -m =n$, and so the above presentation belongs to $\mathcal{P}(n,m)$. This completes the proof of Theorem \ref{SpanThm}.
\end{proof}



\subsection{Main results}\label{densitySec}
Through this section  we fix integers $n, m$, and sets $A, C$, with $|A|=n\geq 2$, and $|C|=m\geq 1$. We make the harmless assumption  that all presentations from $\mc{P}(n,m)$ use the sets $A, C$. Thus, from now on, a group $G$ admits a presentation from $\mc{P}(n,m)$ if and only if there exist integers $\l_{t,i,j}$ such that
\begin{equation}\label{tau2pres2}
G=\langle A, C \mid [a_i, a_j]=\prod_{t=1}^{m} c_t^{\l_t^{ij}},\ [A,C]=[C,C]=1,\ 1\leq i<j\leq m \rangle.
\end{equation}
We denote by $\mc{P}(n,m,\ell)$ the set of  $\tau_2$-presentations from $\mc{P}(n,m)$ with exponents satisfying $|\l_{t,i,j}|\leq \ell$ for all $t,i,j$. 

In this section we prove the two main results of this paper, namely Theorems \ref{mainthm} and \ref{regularitythm}. We start by estimating the number of vectors (with bounded integer  entries) that are linearly independent with a given set of vectors. 
\begin{lemma}\label{countlemma}
Let $V=\{\mathbf{v}_1, \dots, \mathbf{v}_s\}$ be a collection of  $1\times t$ linearly independent vectors with entries in $\mathbb{Z}$ ($s\leq t$). Denote by $I^t$ the set of all $1\times t$ vectors  with entries belonging to a finite set $I\subset \mathbb{Z}$.  Then there are at most $|I|^{|V|}$ vectors from $I^t$ that, together with the $\mathbf{v}_i$'s, form a set of linearly dependent vectors. Consequently, there are at least $|I|^{t} - |I|^{s}$ vectors from $I^t$ that are linearly independent together with the $\mathbf{v}_i$'s. 
\end{lemma}

\begin{proof}
Let $M$ be the $t \times s$ matrix whose $i$-th column consists in the coordinates of $\mb{v}_i$, for $i=1, \dots, a$.  We have $\rank(M)=s$, and thus there are $s$ rows among all rows $\mathbf{r}_1, \dots, \mathbf{r}_t$ of $M$ that form a set of linearly independent $1\times s$ vectors. For notation purposes, we assume these are $\mathbf{r}_1, \dots, \mathbf{r}_s$. For each $k=s+1, \dots, t$,  the vectors $\mathbf{r}_1, \dots, \mathbf{r}_s, \mathbf{r}_k$ are linearly dependent, and hence  there exist rational numbers $q_{k1}, \ldots, q_{ks}$ such that \begin{equation}\label{eq2}\mathbf{r}_{k}= \sum_{i=1}^{s} q_{ki} \mathbf{r}_i.\end{equation} 
By linear independence, the $q_{ki}$'s are unique.
Now  suppose $\mathbf{v}=(n_1, \ldots, n_t)\in I^t$ is a linear combination of the $\mathbf{v}_i$'s, and  let $M'$ be the matrix obtained from $M$ by adding $v$ as a column to the right of $M$. For each $k=1\dots, t$, we have that the $k$-th row of $M'$ is $\mathbf{r}_k' = (\mathbf{r}_k, n_k)$.  Notice that $M$ and $M'$ have the same rank $s$, and  that the first $s$ rows of $M'$  form a set of $1\times (s+1)$ linearly independent vectors, while the remaining rows are a linear combination of the first $s$ rows. By this and by unicity of the $q_{ki}$'s in \eqref{eq2}, $\mathbf{r}_k'= \sum_{i=1}^{s} q_{ki}\mathbf{r}_i'$ for all $k\geq s+1$. In particular,
\begin{equation}\label{eq3}
n_{k}=\sum_{i=1}^{s} q_{ki} n_i \quad \hbox{for all} \quad k\geq s+1.
\end{equation}
It follows that, when choosing $\mathbf{v}$, one has only freedom to specify its first $s$ entries, because the remaining ones are given by \eqref{eq3}. There are  $|I|^{s}$  ways of selecting $s$ integers from $I$, and hence there are at most that many $\mathbf{v}$'s. The last statement of the Lemma follows immediately  from this and from the fact that $|I^t|=|I|^t$. 
\end{proof}

We will need the following notation: for a group $G$ with presentation $P \in \mc{P}(n,m)$ as in \eqref{tau2pres2}, 
%
%
%
%
we denote $\bm{\l}_{i,j}=\left( \l_{1,i,j},\dots, \l_{m,i,j} \right)$ ($i<j$), and we let  $M_{P,k}$ be the following matrices, for $k=1,\dots, n$:
\begin{equation}\label{eq6}
M_{P, k} = \begin{pmatrix}  -\bm{\l}_{1,k}^T & \dots & -\bm{\l}_{k-1, k}^T & \bm{\l}_{k, k+1}^T & \dots & \bm{\l}_{k, n}^T \end{pmatrix},
\end{equation}
here $T$ stands for \emph{transposed}.

We recall that Corollary \ref{superCor} states that any group with presentation \eqref{tau2pres2} is a $\tau_2$-group with Malcev basis $(A;C)$. This fact  will  be used extensively from now on (sometimes implicitly). 
%
%
%
%
Recall also that $\mc{D}(G)$ denotes the Diophantine problem over $G$ (see Definition \ref{DiophDfn}).

We are now ready to prove Theorems \ref{mainthm} and \ref{regularitythm}. 
\begin{thm}\label{mainthm}
Suppose a $\tau_2$-group $G$ is chosen by randomly specifying a presentation $P$ from $\mc{P}(n,m,\ell)$ as in \eqref{tau2pres2}, with $m\geq n-1 \geq 1$. Then the following holds asymptotically almost surely as $\ell \to \infty$: $Z(G)=\C$, all the  $a_k$'s are c-small, and $[a_i, a_j]\neq 1$ for all $i<j$. 

Moreover, a.a.s.: $\mathbb{Z}$ is e-definable in $G$,  $\mc{D}(\mathbb{Z})$ is reducible to $\mc{D}(G)$, $\mathcal{D}(G)$ is undecidable,  $\mathbb{Z}$ is the maximal ring of scalars of $G$, and $G$ is directly indecomposable into non-abelian factors.
\end{thm}
\begin{proof}
Let $G$ and $P$ be as in the statement of the theorem.  Observe that the $m \times (n-1)$ matrices $M_{P,k}$ \eqref{eq6} are precisely the matrices of Corollary \ref{centralizergenerators}. Hence, by Proposition \ref{ThetaCor}, $a_k$ is c-small  provided that $\rank\left(M_{P, k}\right)=n-1$ and that $\l_{t,k,j}\neq0$ for some $t$, $j$.  In this case, by the same proposition, $Z(G)=\C$. Furthermore, since $(A;C)$ is a Malcev basis of $G$,  $[a_i, a_j]\neq 1$ if and only if $\l_{t,i,j}\neq 0$ for some $t$. Thus, to prove the first paragraph of the theorem it suffices to show that, a.a.s.,  $\l_{t,i,j}\neq 0$ for all $t,i,j$, and that $\rank(M_{P, k}) = n-1$ for all $k$.



Note that one may select $P\in \mc{P}(n,m,\ell)$ by first specifying the vectors $\bm{\l}_{1,2}, \dots, \bm{\l}_{1,n}$, then  $\bm{\l}_{2,3}, \dots, \bm{\l}_{2,n}$, then $\bm{\l}_{3,4}, \dots, \bm{\l}_{3,n}$, and so on. In what follows we understand that $P$ is chosen this way. In this case, at the moment when the vector $\bm{\l}_{i,j}$ ($i<j$) is chosen, all vectors $\bm{\l}_{i', j'}$  with $i'<i$ or $i'=i, j'<j$ have already been selected ($i'<j'$).  

We proceed to bound the number of presentations $P\in \mc{P}(n,m,\ell)$ such that $\l_{t,i,j}\neq 0$ for all $t,i<j$, and such that $\rank(M_{P,k})=n-1$ for all $k$.  Denote the set of such $P$'s by $\mc{S}(n,m,\ell)$. 
Observe that $\bm{\l}_{i,j}$ appears only in 
$$
M_{P,i}=\begin{pmatrix} -\bm{\l}_{1,i} & \dots & -\bm{\l}_{i-1,i} & \bm{\l}_{i,i+1} & \dots & \bm{\l}_{i,j} & \dots & \bm{\l}_{i,n} \end{pmatrix}
$$  
(we omit the transpose symbols) and in 
$$
M_{P,j}=\begin{pmatrix} -\bm{\l}_{1,j} & \dots & -\bm{\l}_{i,j} & \dots & -\bm{\l}_{j-1,j} & \bm{\l}_{j,j+1}& \dots & \bm{\l}_{j,n} \end{pmatrix}.
$$  
Denote the vectors preceding $\bm{\l}_{i,j}$ in $M_{P,i}$ and in $M_{P,j}$  by  $\Lambda_{i,j,1}$ and $\Lambda_{i,j,2}$, respectively. 
The method for selecting $P$ described above ensures us that, every time we choose $\bm{\l}_{i,j}$, all the vectors in $\Lambda_{i,j,1}$ and $\Lambda_{i,j,2}$ have already been specified. For this reason, to choose a presentation that belongs to $\mc{S}(n,m,\ell)$, it suffices to select  vectors $\boldsymbol{\l}_{i,j} \in I^m$  in the order explained previously, making sure that, at every step,  $\Lambda_{i,j,1}\cup \{\bm{\l}_{i,j}\}$ and $\Lambda_{i,j,2}\cup \{\bm{\l}_{i,j}\}$ are sets of linearly independent vectors, with  $\Lambda_{i,j,1}$ and $\Lambda_{i,j,2}$ already being linearly independent due to  previous selections. 

Now fix $i<j$. By Lemma \ref{countlemma},  there are at most $L^{|\Lambda_{i,j,1}|}$ vectors $\mathbf{v}$ with entries in $I$ (i.e.\ $\mathbf{v}\in I^m$), such that $\Lambda_{i,j,1} \cup \{\mathbf{v}\}$ is a set of linearly dependent vectors,  where $L=2\ell$, and similarly for $\Lambda_{i,j,2}$. It follows that there are at most $L^{|\Lambda_{i,j,1}|} + L^{|\Lambda_{i,j,2}|}$ vectors $\mathbf{v}\in I^{m}$ such that $\Lambda_{i,j,1} \cup \{\mathbf{v}\}$ or $\Lambda_{i,j,2}\cup \{\mathbf{v}\}$ are linearly dependent. Hence, there are at least $L^m -L^{|\Lambda_{i,j,1}|} -L^{|\Lambda_{i,j,2}|}$ ways to choose $\bm{\l}_{i,j}\in I^m$ such that $\Lambda_{i,j,1}\cup \{\bm{\l}_{i,j}\}$ and $\Lambda_{i,j,2}\cup \{\bm{\l}_{i,j}\}$ are linearly independent. Since $|\Lambda_{i,j,1}|=j-2$ and $|\Lambda_{i,j,2}|=i-1$, 
%
%
$$
|\mc{S}(n,m,\ell)| \geq \prod_{1\leq i<j \leq n} \left(L^m -L^{j-2} -L^{i-1}\right) =p(L).
$$
By hypothesis, $m\geq n-1 > j-2 \geq i-1$ for all $1\leq i<j\leq n$, and hence all factors above are polynomials of degree $m$ with leading coefficient $1$. It follows that $p(L)$ is a polynomial of degree $mn(n-1)/2$, with leading coefficient $1$ as well. 
Notice that $|\mc{P}(n,m,\ell)| = (L+1)^{mn(n-1)/2}$. Therefore  $$\lim_{\ell \to \infty} \frac{|\mc{S}(n,m,\ell)|}{|\mc{P}(n,m,\ell)|}=1.$$  As observed previously, this shows that a.a.s.\ $Z(G)=\C$, all $a_k$'s are c-small,  and  $[a_i, a_j] \neq 1$ for all $i\neq j$.

Finally, observe that the following holds  a.a.s. a) $G$ has a pair of non-commuting c-small elements, namely, $a_1$ and $a_2$ (by what we have proved so far). b) $\mathbb{Z}$ is e-definable in $G$ (by Theorem \ref{mainThmEq2}). c) $\mc{D}(\mathbb{Z})$ is reducible to $\mc{D}(G)$, and $\mc{D}(G)$ is undecidable (by Corollary \ref{RedCor}). d) The maximal ring of scalars of $G$ is $\mathbb{Z}$ (by Theorem \ref{CSmallMaxZ}). e) $G$ is directly indecomposable into non-abelian factors (by Proposition \ref{dirIndec}). This completes the proof  of Theorem \ref{mainthm}.
\end{proof}
%
%
%
%
%
%
%

For the next result, recall that $G$ is  regular if  $Z(G)\leq {\it Is}(G')=\{g\in G\mid g^t \in G' \ \hbox{for some} \ t\in \mathbb{Z}\backslash \{0\}\}$.
\begin{thm}\label{regularitythm}
Let  $G$ be a  $\tau_2$-group obtained by randomly choosing a presentation $P\in \mc{P}(n,m,\ell)$, with $m\geq n-1\geq 1$. Then the following holds a.a.s.\ as $\ell\to \infty$:
\begin{enumerate}
\item If $m\leq n(n-1)/2$, then $G'$ has finite index in $\C$.  If, additionally, $m\geq n-1$, then $G$ is regular.
\item If $m > n(n-1)/2$, then the set $\{[a_i, a_j] \mid i<j\}$ is a basis of $G'$, $G'$ has infinite index in $\C$, and $G$ is not regular. These last two properties hold always, and not only a.a.s.
\end{enumerate}
\end{thm}
\begin{proof}
%
%

As already observed in the proof of Theorem \ref{SpanThm}, the presentation $P$ indicates how to write each commutator $[a_i, a_j]$ as a linear combination of the basis $C$. The matrix $M_P$ consisting on the coefficients of these combinations is precisely the matrix whose rows are the $\bm{\l}_{i,j}$'s. Hence, $\rank(G')= \rank(M_P)$. Let $r=\min\{m, n(n-1)/2\}$ and $L=2\ell+1$. 
We claim that $\rank(M_P)=r$ a.a.s. Indeed, in a similar way as we did in Theorem \ref{mainthm}, and having in mind that $m>r-1$, one  can  repeatedly apply Lemma \ref{countlemma} to see that there are at least 
\begin{equation}\label{eq4}
p(L)=L^m (L^m -L)\dots (L^m - L^{r-1}) L^{m(-r + n(n-1)/2)}
\end{equation}
ways of choosing all the $\bm{\l}_{i,j}$'s (i.e.\ of choosing $P$) so that $\rank(M_P)=r$.  
The degree of the polynomial  $p(L)$ is  $mn(n-1)/2$, and its leading coefficient is $1$. On the other hand, there are $L^{mn(n-1)/2}$ ways of selecting $P\in \mc{P}(n,m,\ell)$. Hence, the probability of chossing $P$ such that $\rank(M_P)=r$ tends to $1$ as $\ell \to \infty$, i.e.\ $\rank(G')=\rank(M_P)=r$ a.a.s.

If $r=m\leq n(n-1)/2$, since $\rank(G')=r = m=\rank(\C)$ a.a.s., and $G'\leq \C$, we obtain that $G'$ has finite index in  $\langle C\rangle$ a.a.s. If, additionally, $r\geq m-1$,  then $Z(G)=\C$  by Theorem \ref{mainthm}, and thus $Z(G)^k\leq G'$, a.a.s., where $k$ is the index of $G'$ in $\C$. This proves that, in this case, $G$ is regular a.a.s., completing the proof of Item 1.

Suppose $m>n(n-1)/2=r$. Then, a.a.s., the elements $\{[a_i, a_j] \mid 1 \leq i< j \leq m\}$ are linearly independent as elements from $\C$, and so they form a basis of $G'$. As we saw in the proof of Theorem \ref{SpanThm}, performing column operations on $M_P$ (interchanging two columns, and adding or subtracting one column to another different column) corresponds to changing the basis $C$ of $\C$ by means of Nielsen transformations. 
Therefore, by choosing a new basis $\bar{C}$, we may assume that the last $m-r$ columns of $M_P$  consist entirely of zeroes. In this case all $[a_i,a_j]$'s belong to  $\langle c_1, \dots, c_{r} \rangle$, and hence  $G'$ has infinite index in $\C$. Moreover, since $C$ is a basis of $\C$, no power of $c_m$ equals a linear combination of the $[a_i,a_j]$'s, which implies that $G$ is not regular. Notice that these last two properties hold always, and not only a.a.s. This completes the proof of Item 2. 
%
 %
%
\end{proof}

%% file: extended.tex
In this section we introduce natural models of   random  polycyclic groups and random finitely generated nilpotent groups, of any nilpotency step and possibly with torsion. These  are  based on choosing random polycyclic or nilpotent presentations, respectively.  We quickly see, however, that  the model of random f.g.\ nilpotent groups   yields finite groups a.a.s. This is the reason why we particularized it to the class of  $\tau_2$-groups in the previous sections. On the other hand, the model of random polycyclic groups yields groups with finite abelianization a.a.s. This could be the starting point of future work on random such groups.

We follow \cite{Handbook1} for the following definitions. A group $G$ is called \emph{polycyclic} if it admits a finite subnormal series with cyclic factors.  A presentation $\langle x_1, \dots, z_n \mid R\ \rangle$ is called a \emph{polycyclic presentation} if there exists a sequence $S=(s_1, \dots, s_n)$ with $s_i\in \mbb{N} \cup \{\infty\}$ and integers $a_{i,k}, b_{i,j,k}, c_{i,j,k}$ such that $R$ consists in the following relations: 
\begin{align}
&x_i^{s_i}= R_{i,i} \ \hbox{with} \ R_{i,i}= x_{i+1}^{a_{i, i+1}} \cdots x_n^{a_{i, n}} \ \hbox{for} \ 1\leq i \leq n \ \hbox{with} \ s_i< \infty, \label{type1}\\
& x_i^{-1}x_j x_i = R_{j,i} \ \hbox{with} \  R_{j,i}=  x_{i+1}^{b_{i,j,i+1}} \cdots x_n^{b_{i,j, n}} \ \hbox{for} \ 1 \leq i < j \leq n, \label{type2}\\
&x_i x_jx_i^{-1}= R_{i,j} \ \hbox{with} \ R_{i,j}= x_{i+1}^{c_{i,j,i+1}} \cdots x_n^{c_{i,j, n}} \ \hbox{for} \ 1 \leq i < j \leq n. \label{type3}
\end{align}
The relations from \eqref{type1} are called \emph{power relations}, and the ones from \eqref{type2} and \eqref{type3} \emph{conjugacy relations}. The presentation above is  a \emph{nilpotent presentation} if 
\begin{align}
&R_{j,i}=  x_j x_{j+1}^{b_{i,j,j+1}} \cdots x_n^{b_{i,j, n}} \label{conj1n} \\
&R_{i,j}= x_j x_{j+1}^{c_{i,j,j+1}} \cdots x_n^{c_{i,j, n}} \label{conj2n}
\end{align}
for all $1 \leq i< j \leq n$.

\begin{lemma}\cite{Handbook1}
Let $G$ be a finitely generated group. Then the following hold:
\begin{enumerate}
	
\item $G$ is polycyclic if and only if it admits a polycyclic presentation.
\item $G$ is nilpotent if and only if it admits a nilpotent presentation.
\item If $G$ is torsion-free and nilpotent, then $G$ admits a nilpotent presentation without power relations, i.e.\  $s_i=\infty$ for all $i=1, \dots, n$.
\end{enumerate}
\end{lemma}

In views of this result, and in a similar fashion to what we did with $\tau_2$-groups, one may model random polycyclic  (or f.g.\ nilpotent, respectively) groups $G$ by fixing a set $X=\{x_1,\dots, x_n\}$ and exponents $S=(s_1, \dots, s_n)$, $s_i \in \mbb{N} \cup \{\infty\}$, and then randomly choosing  the  exponents $a_{i,k}, b_{i,j,k}$ and $c_{i,j,k}$ of a polycyclic (nilpotent) presentation $\langle X\mid R\rangle$. This is done in the same way as for $\tau_2$-groups: the exponents are selected among all integers of magnitude at most $\ell$ (with uniform probability), and $\ell$ is thought of as tending to infinity (see the introduction).  

We assume $n \geq 2$ and $n\geq 3$ for the ``polycyclic and nilpotent models'' above, respectively  - otherwise the resulting groups are always abelian.  $X$ and $S$ remain fixed until the end of the section. The epimorphism that projects from a group onto its abelianization will be denoted by $\ab$.

\begin{lemma}\label{nopol1}
 The model of random nilpotent groups introduced above yields finite groups asymptotically almost surely as $\ell \to \infty$.
\end{lemma}  
\begin{proof}
%
Let $G$ be a group given by randomly choosing the exponents of a nilpotent presentation $\langle X \mid R\rangle$. The conjugacy relations \eqref{conj1n}, \eqref{conj2n} of such presentation state that $[x_j, x_i]=x_j^{-1}R_{j,i}$ and $[x_j, x_i^{-1}]=x_j^{-1}R_{i,j}$ for all $1\leq i< j \leq n$. Thus, in the abelianization of $G$,  $\ab(x_{j}^{-1}R_{j,i})=\ab([x_j,x_i])=1$ and $\ab(x_j^{-1}R_{i,j})=\ab([x_i, x_j])=1$. Taking $j=n-1$ and any $i<j$, $$\ab(x_{n-1}^{-1}R_{n-1, i})= \ab\left(x_n^{b_{i,n-1,n}}\right) =1.$$ It follows that $\ab(x_n)$ has finite order in $G/G'$ a.a.s. Using that $G/G'$ is abelian and a simple induction argument on the identities
$$
\ab(x_j^{-1}R_{j,i})=\ab\left( x_{j+1}^{b_{i,j,j+1}} \cdots x_n^{b_{i,j, n}} \right)=1 \quad \quad (1\leq i<j\leq n)
$$
 we obtain that  $\ab(x_k)$ has finite order for all $k=1, \dots, n$, a.a.s. Since $G/G'$ is abelian and it is generated by the $\ab(x_k)$'s, the whole abelianization $G/G'$ is finite a.a.s. In this case, by  Corollary 9 from Chapter 1 in \cite{Segal}, $G$ is finite.
\end{proof}

We remark that the above result holds even if all  $s_i$'s are fixed to be $\infty$, i.e.\ even if the nilpotent presentations considered have no power relations. 

A very similar approach serves to prove that the corresponding model for polycyclic groups  yields  groups  with finite abelianization a.a.s. This result may be the starting point of future work on random polycyclic groups.

\begin{lemma}\label{nopol2}
The model of random polycyclic  groups introduced above yields groups with finite abelianization a.a.s.	
%
\end{lemma}
\begin{proof}
Let $G$ be a group given by randomly choosing the exponents of a polycyclic presentation $\langle X \mid R\rangle$. The conjugacy relations \eqref{type2}, \eqref{type3} of $R$ in the abelianization of $G$ become $\ab(x_j)=\ab(R_{j,i})$ and $\ab(x_{j})=\ab(R_{i,j})$, for all $1\leq i < j \leq n$.  Taking $i=n-1$ and $j=n$,  $$\ab\left(R_{n,n-1}\right)=\ab\left(x_n^{b_{n-1, n, n}}\right) = \ab\left(x_n^{c_{n-1, n, n}}\right)= \ab\left(R_{n-1, n}\right).$$ It follows that $\ab(x_n)$ has finite order a.a.s. Similarly as before, using induction and the identities $\ab(R_{i,i-1})=\ab(R_{i-1,i})$ for $i=2, \dots, n$, one sees that  $\ab(x_k)$ has finite order in $G/G'$ for all $k=1, \dots, n$, a.a.s., in which  case $G/G'$ is finite because it is abelian and it is generated by $\ab(x_1), \dots, \ab(x_n)$.
\end{proof}